\documentclass[11pt,a4paper,reqno]{amsart}
\usepackage{amsmath,amssymb,amsfonts,epsfig,mathrsfs,cite,lineno,hyperref}
\usepackage[T1]{fontenc}
\usepackage{color}
\usepackage{array}
\usepackage{amsthm}
\usepackage{amstext}
\usepackage{graphicx}
\usepackage{setspace}

\usepackage[title]{appendix}

\makeatletter
\@namedef{subjclassname@2020}{%
  \textup{2020} Mathematics Subject Classification}
\makeatother

\usepackage[margin=2.5cm]{geometry}
\usepackage{color}
\usepackage{enumitem}
\usepackage{amscd,psfrag}
\usepackage{yhmath}
\usepackage[mathscr]{eucal}

\usepackage{comment}

\allowdisplaybreaks[4]

\usepackage{slashed}

\makeatletter
\pdfpageheight\paperheight
\pdfpagewidth\paperwidth

\usepackage{epstopdf}
\usepackage{chngcntr}
\counterwithin{figure}{section}
\usepackage{mathrsfs}

\usepackage{indentfirst}	

\usepackage[normalem]{ulem}
\theoremstyle{plain}
\numberwithin{equation}{section}

\newtheorem{definition}{Definition}[section]
\newtheorem{theorem}[definition]{Theorem}
\newtheorem*{theorem*}{Theorem}

\newtheorem{remark}[definition]{Remark}

\newtheorem*{remark*}{Remark}
\newtheorem*{sideremark*}{Side Remark}

\newtheorem*{claim*}{Claim}
\newtheorem*{q*}{Question}
\newtheorem{lemma}[definition]{Lemma}

\newtheorem*{corollary*}{Corollary}

\newtheorem{dt}[definition]{Definition/Theorem}

\newcommand{\R}{\mathbb{R}}

\newcommand{\na}{\nabla}
\newcommand{\emb}{\hookrightarrow}

\newcommand{\id}{{\rm Id}}
\newcommand{\p}{\partial}

\newcommand{\e}{\varepsilon}

\newcommand{\dd}{{\rm d}}
\newcommand{\G}{\Gamma}

\newcommand{\dvg}{\dd {\rm Vol}_g}

\newcommand{\Exp}{{\bf Exp}}

\newcommand{\tor}{{\bf T}^2}

\newcommand{\diff}{{\bf Diff}}
\newcommand{\sdiff}{{\bf SDiff}}
\newcommand{\hsdiff}{{\bf H}^s \sdiff}

\newcommand{\bra}{\left\langle}
\newcommand{\ket}{\right\rangle}

\newcommand{\T}{{\mathcal{T}}}

\newcommand{\nabar}{{\underline{\na}}}

\newcommand{\pbar}{{\underline{P}}}
\newcommand{\sbar}{{\underline{S}}}

\newcommand{\V}{\mathcal{V}}

\def\XXint#1#2#3{{\setbox0=\hbox{$#1{#2#3}{\int}$ }
\vcenter{\hbox{$#2#3$ }}\kern-.6\wd0}}

\newcommand{\stwo}{{\mathbf{S}^2}}

\newcommand{\met}{{\bra\bullet,\bullet\ket}}

\def\XXint#1#2#3{{\setbox0=\hbox{$#1{#2#3}{\int}$ }
\vcenter{\hbox{$#2#3$ }}\kern-.6\wd0}}

\title{On local solubility of Bao--Ratiu equations on surfaces related to the geometry of diffeomorphism group}

\author{Siran Li}

\address{Siran Li: School of Mathematical Sciences $\&$ CMA-Shanghai, Shanghai Jiao Tong University, No.~6 Science Buildings,
800 Dongchuan Road, Minhang District, Shanghai, China (200240)}

\email{\texttt{siran.li@sjtu.edu.cn}}

\author{Xiangxiang Su}
\address{Xiangxiang Su: School of Mathematical Sciences, Shanghai Jiao Tong University, No.~6 Science Buildings,
800 Dongchuan Road, Minhang District, Shanghai, China (200240)}
\email{\texttt{sjtusxx@sjtu.edu.cn}}


\date{\today}

\begin{document}

\begin{abstract}
We are concerned with the existence of asymptotic directions for the group of volume-preserving diffeomorphisms  of a  closed 2-dimensional surface $(\Sigma,g)$ within the full diffeomorphism group, described by the Bao--Ratiu equations, a system of second-order PDEs introduced in \cite{br, blr}. It is known \cite{p} that asymptotic directions cannot exist globally on any $\Sigma$ with positive curvature. To complement this result, we prove  that asymptotic directions always exist \emph{locally} about a point $x_0 \in \Sigma$ in either of the following cases (where $K$ is the Gaussian curvature on $\Sigma$): (a), $K(x_0)>0$; (b) $K(x_0)<0$; or (c), $K$ changes sign cleanly at $x_0$, \emph{i.e.}, $K(x_0)=0$ and $\na K(x_0) \neq 0$. The key ingredient of the proof is the analysis following Han \cite{h} of a degenerate Monge--Amp\`{e}re equation --- which is of the elliptic, hyperbolic, and mixed types in cases (a), (b), and (c), respectively --- locally equivalent to the Bao--Ratiu equations.  
\end{abstract}
\maketitle

\section{Introduction}\label{sec: intro}

Let $\Sigma$ be a  2-dimensional smooth manifold with a Riemannian metric $g$, and denote by $K$ the Gaussian curvature on $\Sigma$. As we are only concerned with the local solubility of PDEs, without loss of generality, assume throughout this note that $\Sigma$ is closed (\emph{i.e.}, compact and boundaryless).

We establish the local existence of asymptotic directions for the volume-preserving diffeomorphism group $\sdiff(\Sigma,g)$ about a point $x_0 \in \Sigma$, in either of the following cases: 
\begin{itemize}
\item
$\Sigma$ is positively curved about $x_0$, \emph{i.e.}, $K(x_0)>0$;
\item
$\Sigma$ is negatively curved about $x_0$, \emph{i.e.},  $K(x_0)<0$; or
\item
$K$ \emph{changes sign cleanly} at $x_0$, namely that $K(x_0)=0$ and $\na K(x_0) \neq 0$.
\end{itemize}
This is in stark contrast to Palmer's result \cite{p} that any topological $\stwo$ with $K>0$ everywhere admits no global asymptotic directions. Our proof is based on the analysis of a second-order degenerate elliptic PDE of the Monge--Amp\`{e}re type, which is locally equivalent to the  Bao--Ratiu equations introduced in \cite{br, blr}.

\subsection{Extrinsic geometry of volume-preserving diffeomorphism group}
Bao--Ratiu \cite{br} pioneered the research on asymptotic directions of $\sdiff(\Sigma,g)$, the volume-preserving diffeomorphism group on $(\Sigma,g)$, viewed as a submanifold of the full diffeomorphism group $\diff(\Sigma)$. More specifically, define
\begin{align*}
\sdiff(\Sigma,g) = \left\{ \phi \in \diff(\Sigma):\, \phi^\# \dvg = \dvg \right\},
\end{align*}
where $\phi^\#$ is the pullback under $\phi$ and $\dvg$ is the Riemannian volume form on $(\Sigma,g)$. We shall also consider the space $\hsdiff (\Sigma,g)$ of volume-preserving diffeomorphisms with $H^s$-regularity.

Now, following essentially \cite{br}, let us lay out the set up of our note.

Equip $\diff(\Sigma)$ with the $L^2$-based Riemannian metric, which restricts to a right-invariant metric on $\sdiff(\Sigma,g)$. Note that $\diff(\Sigma)$ and $\sdiff(\Sigma,g)$ are infinite-dimensional Lie groups  \cite[\S 2]{blr}. We designate both metrics by $\met$, which induce the Levi-Civita connection on $\sdiff(\Sigma,g)$ and yield an isometric isomorphism 
\begin{equation*}
\G(T\Sigma)\,\, \cong\,\, T_{\eta}\diff(\Sigma)
\end{equation*}
via
\begin{align*}
X \longmapsto X \circ \eta \qquad \text{for any } \eta \in \sdiff(\Sigma,g).
\end{align*}
This together with the Hodge decomposition leads to an $\met$-orthogonal splitting:
\begin{equation}\label{TDiff}
T_\eta \diff(\Sigma) = T_\eta \sdiff(\Sigma,g) \bigoplus \left[T_\eta \sdiff(\Sigma,g)\right]^\perp.
\end{equation}
The components on the right-hand side can be characterised as follows:
\begin{align}\label{TDiff'}
&T_\eta \sdiff(\Sigma,g) := \left\{X\circ \eta:\,X \in \G(T\Sigma) \text{ such that } {\rm div}(X)=0\right\}, \nonumber\\   
&\left[T_\eta \sdiff(\Sigma,g)\right]^\perp = \left\{\na f \circ \eta :\, f:\Sigma\to\R\right\},
\end{align}
where $\na$ is the Riemannian gradient associated with $g$. All the above constructions carry over to $H^s$-vectorfields over $\Sigma$ with $s>2$.

The space of volume-preserving diffeomorphisms on $(\Sigma,g)$ is an important research topic in both global analysis and PDEs for mathematical hydrodynamics. Arnold proved in his seminal 1966 paper \cite{a} that the geodesic equation on $\sdiff(\Sigma,g)$  is precisely the Euler equation describing the motion of an  incompressible inviscid fluid on $(\Sigma,g)$:
\begin{equation}\label{euler, new}
\begin{cases}
\frac{\p v}{\p t} +v \cdot \na v + \na p = 0 \qquad\text{ in } [0,T]\times \Sigma,\\
{\rm div}\, v = 0\qquad\text{ in } [0,T]\times \Sigma,\\
v|_{t=0} = v_0 \qquad \text{ at } \{0\} \times \Sigma
\end{cases}
\end{equation}
for timespan $T>0$, velocity $v \in \G(T\Sigma)$, and pressure $p: \Sigma \to \R$. Equivalently, one may view the Euler equation~\eqref{euler, new} as the horizontal projection of the  
geodesic equation on $\diff(\Sigma,g)$ with respect to the splitting \eqref{TDiff}. If, on the other hand, one considers the \emph{vertical} projection of the  
geodesic equation on $\diff(\Sigma,g)$, the resulting PDE is the Bao--Ratiu equations (see Definition/Theorem~\ref{def/thm} below). We may thus view the study of asymptotic directions as the ``extrinsic counterpart'' to the Euler equation~\eqref{euler, new}. All the discussions in this paragraph can be generalised to higher-dimensional Riemannian manifolds.  See Arnold \cite{a}, Arnold--Khesin \cite{ak}, Bao--Ratiu \cite{br}, Bao--Lafontaine--Ratiu \cite{blr}, and Ebin--Marsden \cite{em}, as well as the references cited therein.

More precisely, let $\nabar$ be the Levi-Civita connection on $\diff(\Sigma)$, and $\pbar$ be the Leray projection of a vectorfield onto its divergence-free (\emph{i.e.}, solenoidal) part. Note that $$(\pbar X) \circ \eta = \pbar_\eta(X\circ \eta),$$ where $$\pbar_\eta: T_\eta\diff(\Sigma) \longrightarrow T_\eta\sdiff(\Sigma,g)$$ is the orthogonal projection in Identity~\eqref{TDiff}. This allows us to define the second fundamental form of $\left(\sdiff(\Sigma,g),\met\right)$ as a submanifold of $\left(\diff(\Sigma),\met\right)$:
\begin{equation}\label{S, def}
\sbar_\eta(X_\eta, Y_\eta) := \big(\na_XY - \pbar(\na_XY)\big)\circ\eta\qquad\text{for any $\eta$, $X$, and $Y$},
\end{equation}
where $\eta \in \sdiff(\Sigma,g)$ and, 
for $X,Y \in \G(T\Sigma)$, we set
\begin{equation}\label{X-eta, Y-eta}
X_\eta = X\circ\eta,\qquad Y_\eta=Y\circ \eta.
\end{equation}

We are at the stage of introducing the following definition and characterisation of asymptotic directions of $\sdiff(\Sigma,g)$, which is the central mathematical object we are interested in.

\begin{dt}[See \cite{blr, br}]\label{def/thm}
A vector field $X_\eta \in T_\eta\sdiff(\Sigma,g)$ is said to be an asymptotic vector/direction if $$\sbar_\eta(X_\eta, X_\eta)=0.$$ For $X \in \G(T\Sigma)$ divergence-free, $X_\eta$ is an asymptotic direction if and only if
\begin{equation}\label{asymp dir eq}
{\rm div}\,X=0\qquad\text{and}\qquad {\rm div}\left(\na_XX\right)=0.
\end{equation}

Equation~\eqref{asymp dir eq} is referred to as the Bao--Ratiu equations.
\end{dt}


\subsection{Main result}

Utilising global geometric arguments (Stokes' theorem and coarea formula, in particular), Palmer proved the following non-existence result \cite[Theorem~1.1]{p}:
\begin{theorem}\label{thm: palmer}
Let $(\Sigma,g)$ be a 2-dimensional compact oriented surface with Gauss curvature $K>0$. Then $\hsdiff(\Sigma,g)$ for $s>4$ admits no global asymptotic directions. 
\end{theorem}

In contrast to Theorem~\ref{thm: palmer}, we prove in this note that there is \emph{no local obstruction to the existence of asymptotic directions on surfaces}. More precisely, our main result is the following:

\begin{theorem}\label{thm: main}
Let $(\Sigma,g)$ be a 2-dimensional Riemannian surface and fix an interior point $x_0 \in \Sigma$. Assume either of the following cases: 
\begin{enumerate}
\item
$\Sigma$ is positively curved about $x_0$, \emph{i.e.}, $K(x_0)>0$;
\item
$\Sigma$ is negatively curved about $x_0$, \emph{i.e.},  $K(x_0)<0$; or
\item
$K$ \emph{changes sign cleanly} at $x_0$, namely that $K(x_0)=0$ and $\na K(x_0) \neq 0$.
\end{enumerate}
Then there exists a solution $X\in \G(T\mathcal{V})$ to the Bao--Ratiu Equation~\eqref{asymp dir eq} in a neighbourhood $\mathcal{V} \subset \Sigma$ of $x_0$. That is, an asymptotic direction exists locally about $x_0$.

More precisely, suppose that the Riemannian metric $g$ is of $H^r$-regularity. In Case~(1), for $r >5$ we obtain $X \in H^{r-3}(\mathcal{V})$; in Case~(2), for $r \geq 8$ we obtain $X \in H^{r-4}(\mathcal{V})$; and in Case~(3), for $r \geq 9$ we obtain $X \in H^{r-5}(\mathcal{V})$.
\end{theorem}

Q. Han \cite{h} established the local existence of isometric embedding $(\Sigma,g)\emb\R^3$ under the ``changing sign cleanly'' condition, by way of proving the local existence of solutions to the Darboux Equation~\eqref{darboux} via transforming it to a symmetrisable hyperbolic system. Our strategy for the proof of Theorem~\ref{thm: main} --- especially for the hyperbolic and mixed-type cases --- relies crucially on the developments in \cite{h}.

For this purpose, we shall rewrite the Bao--Ratiu Equation~\eqref{asymp dir eq} as a possibly degenerate Monge--Amp\`{e}re equation following \cite{br}, and then transform it to a less degenerate equation that is also of the Monge--Amp\`{e}re type. We are indebted to Prof.~Qing Han for suggesting to us the latter transform, which is crucial to our proof.  Our final resulting equation, compared with the ``intermediate product'', bears more resemblance to the Darboux Equation~\eqref{darboux}.

The above procedure will be elaborated in the next section.

\section{A degenerate Monge--Amp\`{e}re equation for asymptotic directions}

\noindent


Now we show that the Bao--Ratiu Equation~\eqref{asymp dir eq} for the asymptotic directions is \emph{locally} equivalent to a second-order nonlinear PDE~\eqref{MA, key} of the Monge--Amp\`{e}re type for a scalar function. By a slight abuse of notations,  we shall also refer to Equation~\eqref{MA, key} as the   Bao--Ratiu equation. 


Suppose that $X \in {\bf H}^s(\Sigma; T\Sigma)$ is an asymptotic direction; $s>2$. Since $X$ is divergence-free, \emph{in a local neighbourhood $\V$ on $\Sigma$} one has that
\begin{equation}\label{J na f}
X=J\na f\qquad \text{for some }   f\in H^{s+1}(\V),
\end{equation}
where $J$ is the almost complex structure on $\Sigma$. In fact, the choice of $f$ can be made \emph{global} whenever the first Betti number of $\Sigma$ is zero. 

Equation~\eqref{J na f} is equivalent to
\begin{align*}
X= -\left(\star\dd f\right)^\sharp,
\end{align*}
where $\star$ is the Hodge star operator, $\sharp$ is the musical isomorphism between $T^*\Sigma$ and $T\Sigma$, and $\dd$ is the exterior differential. In local coordinates $\{\p_1, \p_2\}$ this reads
\begin{align*}
X=(\p_2f)\p_1 - (\p_1f)\p_2.
\end{align*}
Recall that $J(\p_1)=\p_2$ and $J(\p_2)=-\p_1$. 

It is computed in Bao--Ratiu \cite{br} that Equation~\eqref{J na f} and the second identity in Equation~\eqref{asymp dir eq} together yield the following PDE for $f$ of the Monge--Amp\`{e}re type:
\begin{equation*}
\det\left(\partial_i\partial_j f -\G^k_{ij}\partial_k f\right) = \frac{K}{2}\det(g)g^{ij}\partial_i f \partial_j f, 
\end{equation*}
where $\G^k_{ij}$ and $K$ are  the Christoffel symbols and the Gaussian curvature on $(\Sigma,g)$, respectively. This PDE can be expressed in a coordinate-free manner:
\begin{align}\label{MA, key}
\det\left(\na^2 f\right) = \frac{K}{2}\det(g) \left|\na f\right|^2_g.
\end{align}
Here and hereafter, $\na$ is the covariant derivative/Levi-Civita connection on $(\Sigma,g)$. The Riemannian length of a vector relative to $g$ is denoted as $|\bullet|_g$; hence $$g^{ij}\partial_i f \partial_j f\equiv \left|\na f\right|^2_g.$$ The Hessian of $f$ with respect to $g$ is $$\na^2 f \equiv \partial_i\partial_j f -\G^k_{ij}\partial_k f.$$ Einstein's summation convention is assumed throughout.

We refer the reader to the literature on degenerate Monge--Amp\`{e}re equations, including Amano \cite{amano}, Daskalopoulos--Savin \cite{ds}, B. Guan \cite{g}, P. Guan \cite{gpf}, P. Guan--Sawyer \cite{gs}, and C.-S. Lin \cite{l}, amongst others. See also Caffarelli--Nirenberg--Spruck \cite{cns}, Cheng--Yau \cite{cy}, Figalli \cite{f},  Gutiérrez \cite{gu}, and many others for more developments on  Monge--Amp\`{e}re equations. Let us also mention the work \cite{mv} by Moser--Veselov on interesting connections between hydrodynamical models, symplectic geometry, and Monge--Amp\`{e}re equations of varying types.

The Darboux equation for isometric immersions/embeddings of $(\Sigma,g)$ into $\R^3$ reads:
\begin{equation}\label{darboux}
\det\left(\na^2 f\right) = {K}\det(g)\left(1-|\na f|_g^2\right)\qquad \text{ with } |\na f|_g<1.
\end{equation}
It is classically known to be equivalent to the Gauss--Codazzi equations, which are the compatibility equations for the existence of an isometric immersion $(\Sigma,g) \to \R^3$. See the monograph \cite{hh} by Han--Hong and the many references cited therein for a comprehensive survey of isometric immersions and Gauss--Codazzi equations. One cannot fail to notice the resemblance between Equations~\eqref{MA, key} and \eqref{darboux}. However, as noted in Bao--Ratiu \cite[\S 2B, p.63]{br}, such a resemblance may be misleading, for the degeneracy behaviours of these two Monge--Amp\`{e}re-type equations are essentially different. For instance, $f \equiv {\rm constant}$ is a trivial solution to the Bao--Ratiu Equation~\eqref{MA, key}, but it does not satisfy the Darboux Equation~\eqref{darboux} unless $K=0$. 

Nonetheless, as long as only local solutions are concerned, better resemblance between Equations~\eqref{MA, key} and \eqref{darboux} may be unveiled via a simple transform. (We are grateful to Prof.~Qing Han for kindly pointing this out to us.) Consider the simplest case of the Euclidean space: $g_{ij}=\delta_{ij}$ and $\Gamma^k_{ij} \equiv 0$. Shifting to the new variable
\begin{equation} \label{2024.7.11.5}
\tilde{f}(x):=f(x)-x_1
\end{equation}
transforms the Bao--Ratiu Equation~\eqref{MA, key} into
\begin{equation}\label{Monge-Ampere Bao-Ratiu, eucl}
\det \left(\nabla^2 \tilde{f}\right) - \frac{K}{2}\left(1+2\partial_1 \tilde{f}+\left|\nabla \tilde{f}\right|^2\right) = 0,
\end{equation}
which differs from the Darboux Equation~\eqref{darboux} only by the sign in front of $\left|\na \tilde{f}\right|^2$ and the term $2\p_1\tilde{f}$, which is majorised by $1+\left|\na \tilde{f}\right|^2$.  Here $\na$ is the Euclidean gradient.

The general case (\emph{i.e.}, $g_{ij} \neq \delta_{ij}$ and $\Gamma^k_{ij} \neq 0$) appears to be more technically involved, but the idea is essentially the same as above. By considering  $f(x)=x_1+\tilde{f}(x)$ once again, we rewrite the Bao--Ratiu Equation~\eqref{MA, key} as
\begin{align} \label{2024.2.22}
\mathcal{G}\left(\tilde{f}\right)&:=\det\left\{\partial_i\partial_j \tilde{f} -\G^k_{ij}(\delta_{1k}+\partial_k \tilde{f})\right\} \nonumber\\
&\qquad - \frac{K}{2}\det(g)g^{ij}\left(\delta_{1i} \delta_{1j}+\delta_{1i} \partial_j \tilde{f}+\delta_{1j} \partial_i \tilde{f}+\partial_i \tilde{f} \partial_j \tilde{f}\right)\nonumber\\
&=\det\left\{\na_i\na_j \tilde{f}-\G^1_{ij}\right\} - \frac{K}{2}\det(g)\left\{g^{11}+2g^{1k}\p_k\tilde{f}+\left|\na\tilde{f}\right|_g\right\}\nonumber\\
&=0. 
\end{align}
This reduces to Equation~\eqref{Monge-Ampere Bao-Ratiu, eucl} when $g_{ij}=\delta_{ij}$.

Despite the close similarities between Equations~\eqref{2024.2.22} and  \eqref{darboux}, it should be stressed that the presence of extra lower order terms in the former brings about new difficulties in the proof of local solubility, especially in the case that the Gaussian curvature changes sign cleanly. Indeed, among various  issues, the choice of the initial approximate solution $\hat{h}$ in Equation~\eqref{initial approx} below is highly sensitive to these lower order terms.

The following important remark is in order:
\begin{remark} \label{remark: Gamma(0)=0}
To prove the Main Theorem~\ref{thm: main}, it suffices to show the local solubility of the Bao--Ratiu Equation~\eqref{MA, key} of the degenerate Monge--Amp\`{e}re type. Without loss of generality, we may set the marked point $x_0$ in Theorem~\ref{thm: main} to be the origin and consider a neighbourhood of it in $\mathbb{R}^2$. Furthermore, by  selecting a suitable coordinate system $\{x_1, x_2\}$, we may set $$\Gamma_{ij}^k(0)=0.$$
\end{remark}

The proof of our main Theorem~\ref{thm: main} is the content of \S\S\ref{sec, >}--\ref{sec: ><} below.

\section{The case of positive curvature}\label{sec, >}

In this section, we prove Theorem~\ref{thm: main} in the case that the Gaussian curvature $K$ is positive at $x_0=0$. By this assumption one may write
\begin{equation*} 
K(x_1,x_2)=2R+\mathcal{O}\left(|x_1|+|x_2|\right),
\end{equation*}
where
\begin{equation*}
R:=\frac{1}{g^{11}(0)\det\,g(0)}>0.
\end{equation*}
Our goal is to establish the existence of a unique solution for the Bao--Ratiu Equation~\eqref{2024.2.22} under suitable boundary conditions.

We begin by introducing an initial approximate solution:
\begin{equation} \label{2024.2.23.1}
\hat{h}(x_1,x_2)=x_1^2+3x_1 x_2+\frac{5}{2} x_2^2+P_3(x_1,x_2), 
\end{equation}
where $P_3$ is a homogeneous polynomial of degree 3. Note that if we naively choose $P_3=0$, then $\mathcal{G}\left(\hat{h}\right)=\mathcal{O}(|x_1|+|x_2|)$. (Recall the operator $\mathcal{G}$ from Equation~\eqref{2024.2.22}.)

We improve upon the previous naive bound by introducing a nontrivial $P_3$ such that
\begin{equation}\label{G1}
\mathcal{G}\left(\hat{h}\right)=\mathcal{O}\left(x_1^2+x_2^2\right).
\end{equation}
Such $P_3$ can be found via a simple undetermined coefficient argument: take the ansatz $P_3=a x_1^3+b x_1^2 x_2+c x_1x_2^2+d x_2^3$ and substitute it into Equation~\eqref{2024.2.22}. By requiring that the coefficients of $x_1$ and $x_2$ to zero, we obtain an algebraic system for $(a,b,c,d)$. The rank of the coefficient matrix of this system is less than 4, so it has infinitely many solutions. Any such solution for $(a,b,c,d)$ yields a desired $P_3$.

To proceed, set 
\begin{equation} \label{2024.4.22.1}
\begin{cases}
x=\varepsilon^4 \tilde{x},\\
\tilde{f}(x)=\hat{h}(x)+\varepsilon^{11} h(\tilde{x}).
\end{cases}
\end{equation}
The powers $\e^4$ and $\e^{11}$ here are chosen only for convenience --- in particular, to avoid fractional powers --- and are certainly  nonunique. Such a choice ensures the validity of Equation~\eqref{2024.2.23.2}, which will be used to derive the second inequality of \eqref{2024.4.23.1} in the proof of Lemma~\ref{lemma 2.1}.

As $\Gamma_{ij}^k(0)=0$ (see Remark~\ref{remark: Gamma(0)=0}), we obtain that
\begin{align} \label{2024.4.22.2}
\Gamma_{ij}^k(x)&=\Gamma_{ij}^k(0)+\varepsilon^4 \tilde{x} \cdot \left\{D_{x} \Gamma_{ij}^k\right\}(0)+\mathcal{O}(x^2)\nonumber\\
&=\varepsilon^4 \mathcal{O}(\tilde{x})
\end{align}
by Taylor expansion and the scaling $x=\e^4\tilde{x}$ in Equation~\eqref{2024.4.22.1}. Substituting Equations~\eqref{2024.2.23.1}--\eqref{2024.4.22.2} into \eqref{2024.2.22} and arranging terms according to powers of $\varepsilon$, we obtain the PDE for $h$:
\begin{equation} \label{2024.2.23.2}
\varepsilon^3 \left(5\partial_{\tilde{1}\tilde{1}}h-6\partial_{\tilde{1}\tilde{2}}h+2\partial_{\tilde{2}\tilde{2}}h\right)=\varepsilon^6 G\left(\varepsilon,\tilde{x}_1,\tilde{x}_2,\tilde{D}h,\tilde{D}^2 h\right), 
\end{equation}
where $G$ is $C^\infty$ in $\tilde{D}h$ and $\tilde{D}^2 h$, and is $H^{r-4}$ in $\varepsilon,\tilde{x}_1$, and $\tilde{x}_2$. This is because the choice of $P_3$ in Equation~\eqref{2024.2.23.1} depends on up to one derivative of $\G^k_{ij}$ (hence up to two derivatives of $g \in H^r$), and the operator $\mathcal{G}$ in Equation~\eqref{2024.2.22} brings about two more derivatives. Here and hereafter we write $\partial_{\tilde{i}} \equiv \partial_{\tilde{x}_i}$.


Note in passing that, by an inspection on Equations~\eqref{2024.7.11.5}, \eqref{2024.2.23.1}, and \eqref{2024.4.22.1},  our local solution constructed for Equation~\eqref{MA, key} does \emph{not} yield a constant function $f$. This in turn gives us a nonzero asymptotic direction $X=J\na f$; see Equation~\eqref{J na f}.

It remains to prove the local existence of solutions to Equation~\eqref{2024.2.23.2}. Denote by $\Omega$ the unit disc throughout the rest of this section. We first consider the linearised problem, namely Equation~\eqref{2024.4.23.2} below, which is a  standard linear second-order elliptic PDE. 

\begin{lemma} \label{lemma 2.1}
Given any $\underline{h} \in H^m$ with $m > 3$ and $\|\underline{h}\|_{H^m} \le M$, where $M$ is a fixed constant. There exists a unique solution $h \in H^m$ to the boundary value problem:
\begin{equation} \label{2024.4.23.2}
\left\{ 
\begin{array}{cc}
5\partial_{\tilde{1}\tilde{1}}h-6\partial_{\tilde{1}\tilde{2}}h+2\partial_{\tilde{2}\tilde{2}}h=\varepsilon^3 G\left(\varepsilon,\tilde{x}_1,\tilde{x}_2,\tilde{D}\underline{h},\tilde{D}^2 \underline{h}\right) &\text{ in }\Omega,\\
 h=0\qquad  \text{ on } \partial \Omega.
\end{array}
\right.
\end{equation}
Moreover, there is a constant $\varepsilon_0>0$ depending only on $m$ and $M$ such that 
$$
\|h\|_{H^m} \le \frac{M}{2}\qquad \text{ for any } \varepsilon \in(0,\varepsilon_0).
$$
\end{lemma}
\begin{proof}[Proof of Lemma~\ref{lemma 2.1}]
The existence of solution follows from classical elliptic theory \cite{gt}.

Moreover, we bound
\begin{align}  \label{2024.4.23.1}
\|h\|_{H^m} &\le C_m \left\|\varepsilon^3 G\left(\varepsilon,\tilde{x}_1,\tilde{x}_2,\tilde{D}\underline{h},\tilde{D}^2 \underline{h}\right)\right\|_{H^{m-2}} \nonumber\\
&\le C_m \e^3 C_G\|\underline{h}\|_{H^m}\nonumber\\
& \le {M}\slash{2}.
\end{align}
The first line follows from the standard Calder\'{o}n--Zygmund elliptic estimate \cite{gt}. For the second line, note that $G$ equals to the linear combination of products of $\tilde{D}^2 \underline{h}$, $\tilde{D}\underline{h}$, and $\tilde{x}$ with absolutely bounded coefficients. Thanks to the Sobolev inequality $\sum_{j=0}^2\left\|\tilde{D}^j \underline{h}\right\|_{L^\infty} \leq C\|\underline{h}\|_{H^m}$ for $m >3$,  we deduce that $$\|G\|_{H^{m-2}} \leq C_G\|\underline{h}\|_{H^m},$$ where $C_G$ is a uniform constant determined solely by the structure of $G$ (which may depend on $M$). For the final line, we use the assumption $\|\underline{h}\|_{H^m} \le M$ and choose $\e_0$ suitably small.   \end{proof}

The solubility of Equation~\eqref{2024.2.23.2} now follows from a standard fixed point argument.

 Consider the mapping 
$$
\T : H^m \rightarrow H^m, \qquad\qquad \T\underline{h}:=h,
$$ 
where $h$ is the solution to Equation~\eqref{2024.4.23.2} found by Lemma~\ref{lemma 2.1} above. Note that for $m > 3$, one has that $\|\T \underline{h}\|_{H^m} \le M/2$ whenever $\|\underline{h}\|_{H^m} \le M$.

To see that $\T$ is a contraction mapping, denote $$\mathcal{L}[h]:=5\partial_{\tilde{1}\tilde{1}}h-6\partial_{\tilde{1}\tilde{2}}h+2\partial_{\tilde{2}\tilde{2}}h,$$ which is the principal part of Equation~\eqref{2024.4.23.2}. Then
\begin{align*}
\mathcal{L}[h_1]-\mathcal{L}[h_2]&=\mathcal{L}[h_1-h_2]\nonumber\\
&=\varepsilon^3 \left\{G\left(\varepsilon,\tilde{x}_1,\tilde{x}_2,\tilde{D}\underline{h}_1,\tilde{D}^2 \underline{h}_1\right)-G\left(\varepsilon,\tilde{x}_1,\tilde{x}_2,\tilde{D}\underline{h}_2,\tilde{D}^2 \underline{h}_2\right)\right\},
\end{align*}
where $h_i=\T\underline{h}_i$; $i\in \{1,2\}$. Note that $G\left(\varepsilon,\tilde{x}_1,\tilde{x}_2,\tilde{D}\underline{h},\tilde{D}^2 \underline{h}\right)$ is nonlinear in $\underline{h}$.  Arguments similar to those for Equation~\eqref{2024.4.23.1} give us
\begin{align*}
\left\|\T\underline{h}_1-\T\underline{h}_2\right\|_{H^m} &= \|h_1-h_2\|_{H^m} \nonumber\\
 &\le C_m \e^3\left\|\left\{G\left(\varepsilon,\tilde{x}_1,\tilde{x}_2,\tilde{D}\underline{h}_1,\tilde{D}^2 \underline{h}_1\right)-G\left(\varepsilon,\tilde{x}_1,\tilde{x}_2,\tilde{D}\underline{h}_2,\tilde{D}^2 \underline{h}_2\right)\right\}\right\|_{H^{m-2}} \\
& \le \frac{1}{2} \|\underline{h}_1-\underline{h}_2\|_{H^m},
\end{align*}
for $\e=\e(m)$ suitably small.

Therefore, applying the Banach fixed point theorem, we deduce the existence of a solution $h \in H^m$ to Equation~\eqref{2024.2.23.2} with $\|h\|_{H^m} \le M$; $m>3$. In view of Lemma~\ref{lemma 2.1} and that $G=G\left(\varepsilon,\tilde{x}_1,\tilde{x}_2,\tilde{D}\underline{h},\tilde{D}^2 \underline{h}\right)$ is $H^{r-4}$ in $\tilde{x}$, we take $m-2=r-4$ here, namely that 
\begin{equation*}
m=r-2;\qquad r>5.
\end{equation*} 
Thanks to Equation~\eqref{2024.4.22.1}, the above obtained $h \in H^{r-2}$ leads to $\tilde{f} \in H^{r-2}$, which in turn yields the local existence of asymptotic direction $X \in H^{r-3}$ by Equations~\eqref{J na f} and \eqref{2024.7.11.5}.

 This proves the local existence of a nontrivial asymptotic direction  for the case of Gaussian curvature $K(x_0)>0$. 


\section{The case of negative curvature}\label{sec, <}
In this section, we prove Theorem~\ref{thm: main} in the case when the Gaussian curvature $K$ is negative at $x_0=0$. The overall strategy is adapted from Han \cite{h}, but for $K(x_0)<0$ we are in the strictly hyperbolic case, which is simpler than the case of mixed types.

In this case, we express the Gaussian curvature as 
\begin{equation} 
K(x_1,x_2)=2R+\mathcal{O}\left(|x_1|+|x_2|\right), 
\end{equation}
where
$$
R:=-\frac{1}{g^{11}(0)\det\,g(0)}<0.
$$

Now we define the initial approximate solution
\begin{equation} \label{2024.4.24.1}
\hat{h}(x_1,x_2)=\frac{1}{2}x_1^2-\frac{1}{2} x_2^2+P_3(x_1,x_2). 
\end{equation}
(Compare with Equation~\eqref{2024.2.23.1} in the case $K(x_0)>0$.) Note here the  hyperbola $\frac{1}{2}x_1^2-\frac{1}{2} x_2^2$ in $\hat{h}$. As with the elliptic case in \S\ref{sec, >} above, we may choose a nice degree-3-homogeneous polynomial $P_3$ verifying that
\begin{equation}\label{G2}
\mathcal{G}\left(\hat{h}\right)=\mathcal{O}\left(x_1^2+x_2^2\right).
\end{equation}
Recall once again Equation~\eqref{2024.2.22} for the operator  $\mathcal{G}$.

To proceed, adopt the scaling and perturbation constructions as in Equation~\eqref{2024.4.22.1}:
\begin{equation}\label{f-tilde, hyperbolic}
\begin{cases}
x=\varepsilon^4 \tilde{x},\\
\tilde{f}(x)=\hat{h}(x)+\varepsilon^{11} h(\tilde{x}).
\end{cases}
\end{equation}
This together with Equations~\eqref{2024.2.22} and \eqref{2024.4.24.1} yields the \emph{hyperbolic} PDE:
\begin{equation} \label{2024.4.24.2}
\partial_{\tilde{2}\tilde{2}}h-\partial_{\tilde{1}\tilde{1}}h=\varepsilon^3 G\left(\varepsilon,\tilde{x}_1,\tilde{x}_2,\tilde{D}h,\tilde{D}^2 h\right). 
\end{equation}
As in \S\ref{sec, >} above, $G$ is $C^\infty$ in $\tilde{D}h$, $\tilde{D}^2 h$ and is $H^{r-4}$ in $\varepsilon$, $\tilde{x}_1$, and $\tilde{x}_2$. The actual form of $G$ may differ from that in Equation~\eqref{2024.2.23.2}.


The linearised equation for \eqref{2024.4.24.2} is a nonhomogeneous wave equation. We adopt the more suggestive notation $\tilde{x}_1=z, \tilde{x}_2=t$, and take the domain $$\Omega:=\{(z,t) : |z| \le 1, |t| \le 1\}.
$$ The strategy in Han \cite{h} will be followed in the large.

We shall assume as in \cite{h} that there is no $\partial_{tt} h$ on the right-hand side of Equation~\eqref{2024.4.24.2}. Indeed, for $\e>0$ sufficiently small, one may rearrange the term $\partial_{tt} h$ to the left-hand side. Then, by a slight abuse of notations, we may still label the right-hand side of the resulting equation by $\e^3G\left(\varepsilon,\tilde{x}_1,\tilde{x}_2,\tilde{D}h,\tilde{D}^2 h\right)$.

Now we set
\begin{equation} \label{2024.6.12.1}
u:=\p_th\qquad \text{ and }\qquad v:=\p_zh
\end{equation}
and hence rewrite Equation~\eqref{2024.4.24.2} as
\begin{equation} \label{2024.5.25.1}
\left\{ 
\begin{array}{cc}
\partial_t u-\partial_z v=\varepsilon^3 G(\varepsilon,t,z,u,v,\partial_z u,\partial_z v),\\
\partial_t v-\partial_z u=0.
\end{array}
\right.
\end{equation}
Here $G$ is $C^\infty$ in $(u,v,\p_zu, \p_zv)$ and is $H^{r-4}$ in $(\e,t,z)$.

Let us further differentiate Equation~\eqref{2024.5.25.1} with respect to $z$. We obtain
\begin{align} \label{2024.5.25.2}\left\{ 
\begin{array}{cc}
\partial_t \tilde{u}-\varepsilon^3 \partial_{\tilde{u}} G\, \partial_z \tilde{u} -(1+\varepsilon^3 \partial_{\tilde{v}} G)\partial_z \tilde{v}=\varepsilon^3 \tilde G(\varepsilon,t,z,u,v,\tilde{u},\tilde{v}),\\
\partial_t \tilde{v}-\partial_z \tilde{u}=0,
\end{array}
\right.
\end{align}
where
\begin{equation} \label{2024.6.12.2}
\tilde u:=\partial_z u \qquad \text{ and }\qquad \tilde v:=\partial_z v,
\end{equation}
and $$\tilde G=\partial_z G+\tilde{u} \partial_u G +\tilde{v} \partial_v G.$$  Note that $\tilde{G}$ is $C^\infty$ in $(u,v,\p_zu, \p_zv)$ and is $H^{r-5}$ in $(\e,t,z)$.

Define the new variable
\begin{align} \label{2024.6.13.3}
U:=(u, v, \tilde{u}, \tilde{v})^{\top}.
\end{align}
We put together Equation~\eqref{2024.5.25.1} and $1+\varepsilon^3 \partial_{\tilde{v}} G$ times the second identity in Equation~\eqref{2024.5.25.2} to obtain
\begin{align} \label{2024.5.25.3}
A_{\varepsilon}(U) \partial_t U+B_{\varepsilon}(U) \partial_z U=\varepsilon^3 \tilde{G}_{\varepsilon}(U).
\end{align}
The coefficient matrices are given below:
\begin{align} \label{2024.6.11.3}
A_{\varepsilon}(U)=\left(
\begin{array}{cccc}
1 & 0 & 0 & 0 \\
0 & 1 & 0 & 0 \\
0 & 0 & 1 & 0 \\
0 & 0 & 0 & 1+\varepsilon^3 \partial_{\tilde{v}} G(\varepsilon,t,z,U)
\end{array}
\right),
\end{align}
\begin{align} \label{2024.7.2.1}
B_{\varepsilon}(U)= \left(
\begin{array}{cccc}
0 & -1 & 0 & 0 \\
-1 & 0 & 0 & 0 \\
0 & 0 & -\varepsilon^3 \partial_{\tilde{u}} G(\varepsilon,t,z,U) & -1-\varepsilon^3 \partial_{\tilde{v}} G(\varepsilon,t,z,U) \\
0 & 0 & -1-\varepsilon^3 \partial_{\tilde{v}} G(\varepsilon,t,z,U) & 0
\end{array}
\right),
\end{align}
and
$$
\tilde{G}_{\varepsilon}(U)=\left(
\begin{array}{cccc}
G(\varepsilon,t,z,U) \\
0 \\
\tilde G(\varepsilon,t,z,U) \\
0 
\end{array}
\right) .
$$

Consider furthermore the change of variables:
\begin{align} \label{2024.5.25.5}
\tilde{U} := e^{- t} U =e^{- t} (u, v, \tilde{u}, \tilde{v})^{\top}.
\end{align}
Then Equation~\eqref{2024.5.25.3} is equivalent to
\begin{equation}\label{good system}
A_{\varepsilon}(\tilde{U}) \partial_t \tilde{U}+B_{\varepsilon}(\tilde{U}) \partial_z \tilde{U}+C_{\varepsilon}(\tilde{U}) \tilde{U}=\varepsilon^3 \tilde{G}_{\varepsilon}(\tilde{U}),
\end{equation}
where
$$
C_{\varepsilon}(\tilde{U})=\left(
\begin{array}{cccc}
1 & 0 & 0 & 0 \\
0 & 1 & 0 & 0 \\
0 & 0 & 1 & 0 \\
0 & 0 & 0 & 1+\varepsilon^3 \partial_{\tilde{v}} G(\varepsilon,t,z,\tilde{U})
\end{array}
\right).
$$

It is crucial to note that Equation~\eqref{good system} is a positively symmetric first-order hyperbolic system. Indeed, the matrix 
\begin{align} \label{2024.5.27}
\Theta:= C_{\varepsilon}+C^{\top}_{\varepsilon}- \partial_t A_{\varepsilon}- \partial_z B_{\varepsilon}
\end{align}
satisfies
\begin{align*}
\Theta = 
{\rm diag}(2, 2, 2, 2) + \mathcal{O}(\varepsilon^3) > 1.99\, I_4
\end{align*}
for $\e>0$ sufficiently small.

Now we are ready to deduce the existence of a local solution to Equation~\eqref{2024.4.24.2}. We equip Equation~\eqref{good system} with the Cauchy data on $t=-1$:
\begin{equation} \label{2024.5.25.4}
\left\{ 
\begin{array}{cc}
\mathcal{L} \tilde{U}:= A_{\varepsilon}(\tilde{U}) \partial_t \tilde{U}+B_{\varepsilon}(\tilde{U}) \partial_z \tilde{U}+C_{\varepsilon}(\tilde{U}) \tilde{U}=\varepsilon^3 \tilde{G}_{\varepsilon}(\tilde{U}) &\text{ in }\Omega,\\
u=v=\tilde{u}=\tilde{v}=0\qquad \text{ at }t=-1.
\end{array}
\right.
\end{equation}
Recall that $\Omega=[0,1]\times [0,1]$. This Cauchy problem and hence  Equation~\eqref{2024.4.24.2} are solved by applying \cite[Theorems~2.1 and 3.1]{h} as quoted below. Our formulation is slightly more general than what we actually need in this section. The more general form will be invoked in \S\ref{sec: ><} below.

\begin{lemma}[Solubility for positively symmetric hyperbolic system] \label{lemma 3.1}
Given coefficient matrices $A_{\varepsilon}, B_{\varepsilon}, C_{\varepsilon}\in H^m$ with $m \ge 3$ such that the matrix $\Theta$ in Equation~\eqref{2024.5.27} is positive definite; \emph{i.e.}, $\Theta \geq \theta\, {\rm Id}$ for some constant $\theta>0$, and that $A_{\varepsilon}$ is a nonsingular diagonal matrix. Suppose in addition that there exists a constant $\varepsilon_m >0$ such that for any $\varepsilon \in(0,\varepsilon_m)$,
\begin{align} \label{too hot}
\| \partial_z A_{\varepsilon}\|_{H^2}+\| \partial_z B_{\varepsilon}\|_{H^2} \le \eta
\end{align}
holds true for some small $\eta>0$. Then, for any $\tilde{G}_{\varepsilon} \in H^m$, there exists a unique solution $\tilde{U} \in H^m$ to the Cauchy problem 
\begin{equation} \label{2024.7.11}
\left\{ 
\begin{array}{cc}
\mathcal{L} \tilde{U} = A_{\varepsilon}(\tilde{U}) \partial_t \tilde{U}+B_{\varepsilon}(\tilde{U}) \partial_z \tilde{U}+C_{\varepsilon}(\tilde{U}) \tilde{U}=\varepsilon \tilde{G}_{\varepsilon}(\tilde{U}) &\text{ in }\Omega,\\
\tilde{U} \in H_{\mathcal{L}}^1.
\end{array}
\right.
\end{equation}
Here for $A_\e={\rm diag}(\lambda_1, \ldots, \lambda_n)$ we define
$$H_{\mathcal{L}}^1:=\left\{U=(u_1,u_2,u_3,u_4)^{\top} \in H^1: u_i(\cdot, 1)=0 \text { if } \lambda_i<0;\quad u_i(\cdot,-1)=0 \text { if } \lambda_i>0\right\}.$$

 Moreover, there holds for any $s=0,1, \ldots, m$ that
$$
\|\tilde{U}\|_{H^s} \leq c_{s} \|\tilde{G}_{\varepsilon}\|_{H^s}.
$$
The constant $c_s$ depends only on $\theta$, the $H^3$-norms of $A_{\varepsilon}^{-1} B_{\varepsilon}$,  $A_{\varepsilon}^{-1}$, and $C_{\varepsilon}$ for $s \leq 3$, the $H^s$-norms of $A_{\varepsilon}^{-1} B_{\varepsilon}$, $A_{\varepsilon}^{-1}$, and $C_{\varepsilon}$, as well as the $H^{s-1}$-norms of $\partial_z A_{\varepsilon}$ and $\partial_z B_{\varepsilon}$ for $3<s \leq m$.
\end{lemma}

We refer to Han \cite[Appendix]{h} for a proof of this lemma. Despite all its technical details, the idea is, in fact, rather straightforward: one proceeds with standard energy estimates, for which the solution space $H_{\mathcal{L}}^1$ is constructed to ensure that the boundary terms arising from integration by parts are of favourable sign. That is, 
$$
U^{\top} A_{\varepsilon} U(\cdot, 1)-U^{\top} A_{\varepsilon} U(\cdot,-1) \geq 0 \qquad \text { for any } U \in H_{\mathcal{L}}^1 .
$$

In our case, recall from Equation~\eqref{2024.6.11.3} that $A_{\varepsilon}={\rm diag}\left(\lambda_1, \ldots, \lambda_4\right)$ has $\lambda_i>0$ for all $i \in \{1,2,3,4\}$. Hence, $$H^1_{\mathcal{L}}=\left\{U \in H^1:\, u_i \big|_{t=-1}=0 \text{ for each }i \in \{1,2,3,4\}\right\}.$$ This explains why we consider the Cauchy problem~\eqref{2024.5.25.4} here. In addition, directly from the definition of $A_\e$ and $B_\e$ in  Equations~\eqref{2024.6.11.3} and \eqref{2024.7.2.1}, it is easy to verify the condition~\eqref{too hot}.

Therefore, we are in the situation of applying Lemma~\ref{lemma 3.1} to deduce the existence of solution $\tilde{U}$ to Equation~\eqref{2024.5.25.4}. Note that the coefficient matrices $A_\e$ and $B_\e$ (see Equations~\eqref{2024.6.11.3}, \eqref{2024.7.2.1}) are $H^{r-4}$ in $\tilde{x}$ and the source term $\tilde{G}_\e$ is $H^{r-5}$ in $\tilde{x}$, so 
\begin{equation*}
m=r-5;\qquad r\geq 8
\end{equation*}
in light of the conditions in Lemma~\ref{lemma 3.1}. The previous arguments in this section yield a solution $\tilde{U}$ in $H^{r-5}$. By the definition of $\tilde{U}$ in Equation~\eqref{2024.5.25.5}, we find that the second-order derivatives of $h$ are in $H^{r-5}$, hence $h \in H^{r-3}$ in a neighbourhood of $x_0$. (Here we also use the uniqueness of solution to deduce that $\tilde{u}=\p_zu=\p_tv$ and $\tilde{v}=\p_zv$.) From Equation~\eqref{f-tilde, hyperbolic} we then infer that $\tilde{f} \in H^{r-3}$. This in turn yields the local existence of asymptotic direction $X \in H^{r-4}$ by Equations~\eqref{J na f} and \eqref{2024.7.11.5}. 

This proves the local existence of a nontrivial asymptotic direction for the case of Gaussian curvature $K(x_0)>0$.

\section{The case of Gaussian curvature changing sign cleanly}\label{sec: ><}

In this section, we prove Theorem~\ref{thm: main} in the case where the Gauss curvature $K$ changes sign cleanly at $x_0=0$. That is, 
\begin{equation*}
K(0)=0\qquad\text{and}\qquad\na K(0) \neq 0.
\end{equation*}

The overall strategy is parallel to \S\ref{sec, <} above, which is essentially an adaptation of \cite{h}. In particular, the crux is to reduce the problem to a  positively symmetric first-order hyperbolic PDE system for the variable $\tilde{U}$ as in Equation~\eqref{2024.5.25.5}. Nevertheless, the mixed-type features for the Bao--Ratiu equation in the changing sign cleanly case lead to considerable additional technical difficulties. Among other issues, in order to warrant the positivity of the resulting first-order hyperbolic system, we need to judiciously select all the free variables up to degree 3 in the initial approximate polynomial $\hat{h}$; see Equation~\eqref{initial approx} below.

To begin with, as $K$ changes sign cleanly at $x_0=0$,  we may choose local coordinates $\{x_1,x_2\}$ about $0$ such that
\begin{equation} \label{2024.7.11.2}
K(x_1,x_2)=2R x_2+\mathcal{O}\left(x_1^2+x_2^2\right),
\end{equation}
by virtue of the Taylor expansion. The constant $R>0$ is to be determined later. In fact, by a scaling of the form $(x_1, x_2) \mapsto (\lambda x_1, \lambda x_2)$ or, equivalently, a scaling of the Riemannian metric $g \mapsto \lambda^{-2}g$ for  suitable $\lambda>0$, we have the freedom of choosing $R$ to be sufficiently large. 

Now let us construct the initial approximate solution $\hat{h}$, which is the counterpart to Equations~\eqref{2024.2.23.1} and \eqref{2024.4.24.1} in \S\ref{sec, >} and \S\ref{sec, <}, respectively. We take the ansatz:
\begin{align} \label{initial approx}
\hat{h}(x_1,x_2)&=\frac{\gamma^2}{2}x_1^2+\gamma x_1 x_2+\frac{1}{2} x_2^2\nonumber\\
&\qquad\qquad+a x_1^3+b x_1^2 x_2+c x_1x_2^2+d x_2^3+P_4(x_1,x_2)
\end{align}
where $\gamma$, $a$, $b$, $c$, and $d$ are constants, and $P_4$ is a degree-4-homogeneous polynomial in $x_1$ and $x_2$. Recall the operator $\mathcal{G}$ from Equation~\eqref{2024.2.22}. We shall specify these undetermined coefficients to ensure that 
\begin{equation} \label{2024.6.11.1}
\mathcal{G}\left(\hat{h}\right)=\mathcal{O}\left(\left(x_1^2+x_2^2\right)^{{3}/{2}}\right).
\end{equation}

To this end, recall from Remark~\ref{remark: Gamma(0)=0} that $\Gamma_{ij}^k(0)=0$. Taylor expansion yields that
\begin{equation*} 
\left\{ 
\begin{array}{cc}
\Gamma_{11}^1(x)=s_1 x_1 +t_1 x_2 +\mathcal{O}\left(x_1^2+x_2^2\right),\\
\Gamma_{12}^1(x)=s_2 x_1 +t_2 x_2 +\mathcal{O}\left(x_1^2+x_2^2\right),\\
\Gamma_{22}^1(x)=s_4 x_1 +t_4 x_2 +\mathcal{O}\left(x_1^2+x_2^2\right),
\end{array}
\right.
\end{equation*}
where $s_1$, $t_1$, $s_2$, $t_2$, $s_4$, and $t_4$ are real constants determined solely by the fixed Riemannian metric $g$. For reasons that shall be made clear in  subsequent developments, we select $(a,b,c,d)$ as follows:
\begin{equation} \label{2024.6.14.1}
\begin{cases}
a=\frac{1}{6}s_1,\\
b=\frac{1}{2}s_2,\\
c=\frac{1}{2}s_4,\\
\gamma^2(6d-t_4)-2\gamma(2c-t_2)+(2b-t_1)=R g^{11}(0)\det(g)(0).
\end{cases}
\end{equation}
In the final line we set $\gamma$ and $R$ free and specify $d$. Arguments analogous to those leading to Equation~\eqref{G1} in \S\ref{sec, >} and  Equation~\eqref{G2} in \S\ref{sec, <} allow us to achieve  \eqref{2024.6.11.1} by selecting $P_4$ appropriately. Indeed, for each given pair $(\gamma,R)$ we may choose $P_4$ to eliminate the quadratic terms arising from the expansion of $\mathcal{G}\left(\hat{h}\right)$. Thus, by now, the free parameters awaiting further specification in the initial approximate solution~\eqref{initial approx} are $\gamma$ and $R$ only.


Next we introduce 
\begin{equation}\label{2024.7.11.3}
\begin{cases}
x=\varepsilon^2 \tilde{x},\\
\tilde{f}(x)=\hat{h}(x)+\varepsilon^{7} h(\tilde{x}).
\end{cases}
\end{equation}
The choice here is different from Equations~\eqref{2024.4.22.1} and \eqref{f-tilde, hyperbolic} for the cases $K(0)>0$ and $K(0)<0$, respectively. This is because in the case of Gaussian curvature changing sign cleanly, Equation~\eqref{2024.7.11.2} for $K$ has no constant terms (\emph{i.e.}, order-zero terms in $x$). Trial and error shows that the powers $\e^2$ and $\e^7$ in Equation~\eqref{2024.7.11.3} turns out  favourable for energy estimates. 

To proceed, as in \S\ref{sec, <} we designate
\begin{align*}
 \tilde{x}_1=z\qquad\text{and}\qquad \tilde{x}_2=t,
\end{align*}
as well as 
\begin{align*}
u=\p_th \qquad\text{and}\qquad v=\p_zh.
\end{align*}
Substituting the choices of coordinates $\{x_1,x_2\}$ and functions $\hat{h}$, $\tilde{f}$ (see Equations~\eqref{2024.7.11.2}, \eqref{initial approx}, \eqref{2024.6.14.1}, and \eqref{2024.7.11.3}) into the Bao--Ratiu Equation~\eqref{2024.2.22} of the degenerate Monge--Amp\`{e}re form, we recast Equation~\eqref{2024.2.22} into the following PDE system: 
\begin{equation} \label{2024.6.13.1}
\left\{ 
\begin{array}{rl}
\frac{1}{\varepsilon^2} [\gamma^2+(2b-t_1)\varepsilon^2 t]\partial_{t} u-\frac{2}{\varepsilon^2}[\gamma+(2c-t_2)\varepsilon^2 t]\partial_{z} u\\
+\frac{1}{\varepsilon^2}[1+(6d-t_4)\varepsilon^2 t]\partial_{z} v&=\varepsilon G(\varepsilon,t,z,u,v,\partial_z u,\partial_z v),\\
\partial_t v-\partial_z u&=0.
\end{array}
\right.
\end{equation}

Here, by an abuse of notations, we again denote the source term in Equation~\eqref{2024.6.13.1} as $G$. It is slightly different from the same symbols in \S\S\ref{sec, >} and \ref{sec, <} (see Equations~\eqref{2024.2.23.2} and \eqref{2024.4.24.2}). Indeed, our $G$ here is $C^\infty$ in $\tilde{D}h$, $\tilde{D}^2 h$ and is $H^{r-5}$ in $\varepsilon$, $\tilde{x}_1$, and $\tilde{x}_2$: the choice of $P_4$ in Equation~\eqref{initial approx} depends on up to two derivatives of $\G^k_{ij}$ (hence three derivatives of $g \in H^r$), and the operator $\mathcal{G}$ in Equation~\eqref{2024.2.22} involves two more derivatives.

Reasoning as in the paragraph right above Equation~\eqref{2024.6.12.1} in \S\ref{sec, <}, we assume from now on that $G$ contains no $\p_tu$.


Next, as in Equation~\eqref{2024.6.12.2} , we set 
\begin{equation*}
\tilde u:=\partial_z u \qquad \text{and}\qquad \tilde v:=\partial_z v.
\end{equation*}
Differentiation of Equation~\eqref{2024.6.13.1} with respect to $z$ yields that
\begin{equation} \label{2024.6.13.2}
\left\{ 
\begin{array}{rr}
\frac{1}{\varepsilon^2}[\gamma^2+(2b-t_1)\varepsilon^2 t]\partial_{t} \tilde u-\frac{2}{\varepsilon^2}[\gamma+(2c-t_2)\varepsilon^2 t+\varepsilon^3 \partial_{\tilde u} G]\partial_{z} \tilde u\\
+\frac{1}{\varepsilon^2}[1+(6d-t_4)\varepsilon^2 t-\varepsilon^3 \partial_{\tilde v} G]\partial_{z} \tilde v&=\varepsilon \tilde G(\varepsilon,t,z,u,v,\tilde u,\tilde v),\\
\partial_t \tilde v-\partial_z \tilde u=0,
\end{array}
\right.
\end{equation}
where
 $$\tilde G := \partial_z G+\tilde{u} \partial_u G +\tilde{v} \partial_v G.$$
Note that $\tilde{G}$ is $C^\infty$ in $(u,v,\tilde{u},\tilde{v})$ and is $H^{r-6}$ in $(\e,t,z)$.

Now we perform a  change of variables different from Equation~\eqref{2024.5.25.5} in the case $K(x_0)<0$. Denote as in Equation~\eqref{2024.6.13.3} that
\begin{equation*}
U=(u, v, \tilde{u}, \tilde{v})^{\top}.
\end{equation*} 
Then we put
\begin{equation}\label{yyy}
\tilde{U} := e^{- \varepsilon^2 t} U.
\end{equation}
The advantage for introducing this new variable is manifested in Equation~\eqref{xxx} below.

 Equations~\eqref{2024.6.13.1} and \eqref{2024.6.13.2} together are equivalent to the following PDE system:
$$
A_{\varepsilon}(\tilde{U}) \partial_t \tilde{U}+B_{\varepsilon}(\tilde{U}) \partial_z \tilde{U}+C_{\varepsilon}(\tilde{U}) \tilde{U}=\varepsilon \tilde{G}_{\varepsilon}(\tilde{U}),
$$
where
\begin{small}
\begin{align} \label{2024.6.14.5}
A_{\varepsilon}(\tilde{U})=\frac{1}{\varepsilon^2}\left(
\begin{array}{cccc}
\gamma^2+(2b-t_1)\varepsilon^2 t & 0 & 0 & 0 \\
0 & -[1+(6d-t_4)\varepsilon^2 t] & 0 & 0 \\
0 & 0 & \gamma^2+(2b-t_1)\varepsilon^2 t & 0 \\
0 & 0 & 0 & -[1+(6d-t_4)\varepsilon^2 t-\varepsilon^3 \partial_{\tilde v} G]
\end{array}
\right),
\end{align}
\end{small}
\begin{footnotesize}
\begin{align}\label{B, new}
B_{\varepsilon}(\tilde{U})=\frac{1}{\varepsilon^2} \left(
\begin{array}{cccc}
-2[\gamma+(2c-t_2)\varepsilon^2 t] & 1+(6d-t_4)\varepsilon^2 t & 0 & 0 \\
1+(6d-t_4)\varepsilon^2 t & 0 & 0 & 0 \\
0 & 0 & -2[\gamma+(2c-t_2)\varepsilon^2 t+\varepsilon^3 \partial_{\tilde u} G] & 1+(6d-t_4)\varepsilon^2 t-\varepsilon^3 \partial_{\tilde v} G \\
0 & 0 & 1+(6d-t_4)\varepsilon^2 t-\varepsilon^3 \partial_{\tilde v} G & 0
\end{array}
\right),
\end{align}
\end{footnotesize}
\begin{small}
\begin{align}
C_{\varepsilon}(\tilde{U})=\left(
\begin{array}{cccc}
\gamma^2+(2b-t_1)\varepsilon^2 t & 0 & 0 & 0 \\
0 & -[1+(6d-t_4)\varepsilon^2 t] & 0 & 0 \\
0 & 0 & \gamma^2+(2b-t_1)\varepsilon^2 t & 0 \\
0 & 0 & 0 & -[1+(6d-t_4)\varepsilon^2 t-\varepsilon^3 \partial_{\tilde v} G]
\end{array}
\right),
\end{align}
\end{small}
and
\begin{align} \label{2024.6.14.6}
\tilde{G}_{\varepsilon}(\tilde{U})=\left(
\begin{array}{cccc}
G(\varepsilon,t,z,\tilde{U}) \\
0 \\
\tilde G(\varepsilon,t,z,\tilde{U}) \\
0 
\end{array}
\right).
\end{align}

We shall conclude the proof by employing Lemma~\ref{lemma 3.1}. For this purpose, it is crucial to check the positive definiteness of the matrix  
\begin{align*}
\Theta := C_{\varepsilon}+C^{\top}_{\varepsilon}- \partial_t A_{\varepsilon}- \partial_z B_{\varepsilon} 
\end{align*}
defined as in Equation~\eqref{2024.5.27}. Straightforward computation yields that
\begin{small}
\begin{align} \label{xxx}
\Theta= \left(
\begin{array}{cccc}
2\gamma^2-(2b-t_1) & 0 & 0 & 0 \\
0 & (6d-t_4)-2 & 0 & 0 \\
0 & 0 & 2\gamma^2-(2b-t_1) & 0 \\
0 & 0 & 0 &(6d-t_4)-2 
\end{array}
\right) + \mathcal{O}(\varepsilon).
\end{align}
\end{small}
Thus, it suffices to require \begin{equation} \label{2024.6.14.2}
2\gamma^2-(2b-t_1)>0 \qquad\text{and}\qquad (6d-t_4)-2>0
\end{equation}
for $\Theta$ to be positive definite.

In view of the choice of $a$, $b$, $c$, and $d$ in Equation~\eqref{2024.6.14.1}, the conditions in Equation~\eqref{2024.6.14.2} are equivalent to 
\begin{equation} \label{2024.6.14.3}
\begin{cases}
2\gamma^2-(s_2-t_1)>0,\\
\frac{2}{\gamma}(s_4-t_2)-\frac{1}{\gamma^2}(s_2-t_1)-2+\frac{1}{\gamma^2}R g^{11}(0)\det\,g(0)>0.
\end{cases}
\end{equation} 
The first inequality is satisfied when $\gamma>0$ is chosen to be sufficiently large, depending on $g$ only. With such a $\gamma$ fixed once and for all, let us specify $R$ to meet the second condition. Indeed, observe that  
\begin{align*}
&\frac{2}{\gamma}(s_4-t_2)-\frac{1}{\gamma^2}(s_2-t_1)-2+\frac{1}{\gamma^2}R g^{11}(0)\det\,g(0)\\
&\qquad\qquad\qquad >R g^{11}(0)\det\,g(0)\frac{1}{\gamma^2}+2(s_4-t_2)\frac{1}{\gamma}-4.
\end{align*}
As $g^{11}(0)>0$ and $\det\,g >0$, one may choose a sufficiently large $R$ to ensure the positivity of the right-hand side. It is crucial to note here that the choice of $R$, which is tantamount to the choice of $\lambda=\lambda(R)$ in $(x_1, x_2) \mapsto (\lambda x_1, \lambda x_2)$ or in the scaling of the Riemannian metric $g \mapsto \lambda^{-2}g$, does not alter the values of $s_i$ and $t_i$ ($i\in\{1,2,4\}$), \emph{i.e.}, the first-order terms in the Christoffel symbols $\G^k_{ij}(x)$. In fact, a simple exercise in Riemannian geometry shows that $g \mapsto \lambda^{-2}g$ leaves invariant $\G^k_{ij}$. Alternatively, write for arbitrary $\mu>0$ that
\begin{equation*}
\Gamma(x) \equiv \bar{\Gamma}(\tilde{x})\qquad\text{where}\qquad x:=\mu\tilde{x}.
\end{equation*}
By Taylor expansion we have
\begin{align*}
\Gamma(x)&=\Gamma(\mu \tilde{x}) \\
&=\Gamma(0)+x \cdot D_{x} \Gamma(0)+\mathcal{O}(x^2)\\
&=\Gamma(0)+\tilde{x} \cdot \left\{D_{\tilde{x}} \Gamma(\mu \tilde{x})\Big|_{\tilde{x}=0}\right\}+\mathcal{O}\left(\tilde{x}^2\right)
\end{align*}
and
\begin{align*}
\bar{\Gamma}(\tilde{x}) =\bar{\Gamma}(0)+\tilde{x} \cdot D_{\tilde{x}} \bar{\Gamma}(0)+\mathcal{O}\left(\tilde{x}^2\right).
\end{align*}
Thus the assertion follows by comparing the terms in $\Gamma(x)$ and $\bar{\G}(\tilde{x})$.


We are now ready to conclude. For $$\Omega=[0,1]\times [0,1]$$ as in \S\ref{sec, <}, let us consider the following  problem:
\begin{equation} \label{new, eq}
\left\{ 
\begin{array}{cc}
\mathcal{L} \tilde{U}:= A_{\varepsilon}(\tilde{U}) \partial_t \tilde{U}+B_{\varepsilon}(\tilde{U}) \partial_z \tilde{U}+C_{\varepsilon}(\tilde{U}) \tilde{U}=\varepsilon \tilde{G}_{\varepsilon}(\tilde{U}) &\text{ in }\Omega,\\
{u=\tilde{u}=0 \text{ on }t=-1,\qquad v=\tilde{v}=0 \text{ on }t=1}.
\end{array}
\right.
\end{equation}
The coefficient matrices $A_{\varepsilon}$, $B_{\varepsilon}(\tilde{U})$, $C_{\varepsilon}(\tilde{U})$, and $\tilde{G}_{\varepsilon}(\tilde{U})$ are defined in Equations~\eqref{2024.6.14.5}--\eqref{2024.6.14.6}. One readily deduces from  Lemma~\ref{lemma 3.1}  the existence and uniqueness of the solution $\tilde{U} \in H^m$ to Equation~\eqref{new, eq}, which in turn leads to a solution to the Bao--Ratiu Equation~\eqref{2024.2.22}. Note by Equations~\eqref{2024.7.11.5}, \eqref{initial approx}, and \eqref{2024.7.11.3} that  $\tilde{f}={\rm constant}$ is not a solution to  Equation~\eqref{2024.2.22}.

Due to the regularity of $G$ and $\tilde{G}$, the coefficient matrices $A_\e$, $B_\e$ defined in Equations~\eqref{2024.6.14.5} and \eqref{B, new} are $H^{r-5}$ in $\tilde{x}$, while the source term $\tilde{G}_\e$ in Equation~\eqref{2024.6.14.6} is $H^{r-6}$ in $\tilde{x}$. Hence, in view of the conditions in Lemma~\ref{lemma 3.1}, we choose
\begin{equation*}
m=r-6;\qquad r\geq 9.
\end{equation*}
The previous arguments in this section yield a solution $\tilde{U}$ in $H^{r-6}$. By the definition of $\tilde{U}$ in Equation~\eqref{yyy}, we find that the second-order derivatives of $h$ are in $H^{r-6}$, hence $h \in H^{r-4}$ in a neighbourhood of $x_0$. (Here we also use the uniqueness of solution to deduce that $\tilde{u}=\p_zu=\p_tv$ and $\tilde{v}=\p_zv$.) From Equation~\eqref{2024.7.11.3} we then infer that $\tilde{f} \in H^{r-4}$. This in turn yields the local existence of asymptotic direction $X \in H^{r-5}$ by Equations~\eqref{J na f} and \eqref{2024.7.11.5}.

This completes the proof of the local existence of a nontrivial asymptotic direction in the case that the Gaussian curvature $K$ changing sign cleanly at $x_0$.

\medskip

\noindent
{\bf Acknowledgement}. Both authors are greatly indebted to Professor Qing Han for insightful discussions and generous sharing of ideas. The research of SL is supported by NSFC Projects 12201399, 12331008, and 12411530065, Young Elite Scientists Sponsorship Program by CAST  2023QNRC001, the National Key Research $\&$ Development Program 2023YFA1010900, and the Shanghai Frontier Research Institute for Modern Analysis. The research of XS is partially supported by the National Key Research $\&$ Development Program 2023YFA1010900. 

\medskip
\noindent
{\bf Competing Interests Statement}. We declare that there is no conflict of interests involved.


\begin{thebibliography}{99}

\bibitem{amano}
K. Amano, The Dirichlet problem for degenerate elliptic 2-dimensional Monge--Amp\`{e}re equation, \textit{Bull. Austral. Math. Soc.} \textbf{37} (1988), 389--410.

\bibitem{a}
V.~I. Arnold, Sur la géométrie différentielle des groupes de Lie de dimension infinie et ses applications à l'hydrodynamique des fluides parfaits. \textit{Ann. Inst. Fourier (Grenoble)} \textbf{16} (1966), 319--361.

\bibitem{ak}
V.~I. Arnold and B.~A. Khesin, \textit{Topological methods in hydrodynamics}. Second edition. Applied Mathematical Sciences, 125. Springer, Cham, 2021.

\bibitem{blr}
D. Bao, J. Lafontaine, and T. Ratiu, On a non-linear equation related to the geometry of the diffeomorphism group, \textit{Pacific J. Math.} \textbf{158} (1993), 223--242.

\bibitem{br}
D. Bao and T. Ratiu, On the geometric origin and the solvability of a degenerate Monge--Ampere equation, \textit{Proc. Symp. Pure Math.} \textbf{54} (1993), 55--69.

\bibitem{cns}
L. Caffarelli, L. Nirenberg, and J. Spruck, The Dirichlet problem for nonlinear second-order elliptic equations. I. Monge-Ampère equation, \textit{Comm. Pure Appl. Math.} \textbf{37} (1984), no.3, 369--402.

\bibitem{cy}
S.~Y. Cheng and S.~T. Yau, On the regularity of the Monge-Ampère equation $\det\left({\partial^2 u}/{\partial x_i \partial x_j}\right) = F(x,u)$, \textit{Comm. Pure Appl. Math.} \textbf{30} (1977), no. 1, 41--68.

\bibitem{ds}
P. Daskalopoulos and O. Savin, On Monge-Ampère equations with homogeneous right-hand sides,  \textit{Comm. Pure Appl. Math.} \textbf{62} (2009), no. 5, 639--676.

\bibitem{em}
D.~G. Ebin and J. Marsden, Groups of diffeomorphisms and the motion of an incompressible fluid, \textit{Ann. of Math.} (2) \textbf{92} (1970), 102--163.



\bibitem{f}
A. Figalli, The Monge-Ampère Equation and its Applications, Zur. Lect. Adv. Math., EMS Press, Zürich, 2017.

\bibitem{gt}
D. Gilbarg and N.~S. Trudinger, \textit{Elliptic partial differential equations of second order}, Vol. 224. No. 2. Berlin: springer, 1977.

\bibitem{g}
B. Guan, The Dirichlet problem for Monge-Ampère equations in non-convex domains and spacelike hypersurfaces of constant Gauss curvature, \textit{Trans. Amer. Math. Soc.} \textbf{350} (1998), no 12, 4955--4971.

\bibitem{gpf}
P. Guan, Regularity of a class of quasilinear degenerate elliptic equations, \textit{Adv. Math.} \textbf{132} (1997), 24--45.

\bibitem{gs}
P. Guan and E. Sawyer, Regularity of subelliptic Monge–Ampère equations in the plane, \textit{Trans. Amer. Math. Soc.} \textbf{361} (2009), 4581--4591. 

\bibitem{gu}
C. E. Gutiérrez, The Monge-Ampère equation. Progress in Nonlinear Differential Equations and Their Applications, 44. Birkhäuser, Boston, 2001.

\bibitem{h}
Q. Han, On the isometric embedding of surfaces with Gauss curvature changing sign cleanly, \textit{Comm. Pure Appl. Math.} \textbf{58} (2005), 285--295.

\bibitem{hh}
Q. Han and J.-X. Hong, \textit{Isometric Embedding of Riemannian Manifolds in Euclidean Spaces}. American Mathematical Society, Providence 2006.

\bibitem{l}
C. S. Lin, The local isometric embedding in $\R^3$ of 2-dimensional Riemannian manifolds with nonnegative curvature, \textit{J. Differential Geom.} \textbf{21} (1985), no. 2, 213--230.

\bibitem{mv}
J. Moser and A.~P. Veselov, Two-dimensional `discrete hydrodynamics' and Monge--Amp\`{e}re equations, \textit{Ergod. Th. $\&$ Dynam. Sys.} \textbf{22} (2002), 1575--1583.

\bibitem{p}
B. Palmer, The Bao--Ratiu equations on surfaces, \textit{Proc. R. Soc. Lond. A} \textbf{449} (1995), 623--627.




\end{thebibliography}
\end{document}